\documentclass[12pt]{article}
\usepackage[letterpaper, left=.9in, top=.8in, right=.9in, bottom=.4in,nohead,includefoot,verbose, ignoremp]{geometry}
\pdfoutput=1

\usepackage{booktabs}
\usepackage[square,sort,comma,numbers]{natbib}
\usepackage{times} 
\usepackage{latexsym,epsfig,amssymb,amsmath,amsfonts,graphicx,amsthm,ifthen,pifont,comment,enumerate,setspace,multirow,color,footnote}
\usepackage{amscd, xy}
\input{xy}
\xyoption{all}

\def\begmat{\left(\begin{array}}\def\endmat{\end{array}\right)}

\def\bi{\begin{itemize}\setlength{\itemsep}{0pt}} \def\ei{\end{itemize}}

\def\bl{\begin{list}{\labelitemi}{\leftmargin=1em}\setlength{\itemsep}{-2.5pt}}  \def\el{\end{list}}
\def\bn{\begin{enumerate}} \def\en{\end{enumerate}}
\def\bt{\begin{table}[h]} \def\et{\end{table}}
\def\bc{\begin{center}} \def\ec{\end{center}}
\def\T{{ \mathrm{\scriptscriptstyle T} }}

\newcommand{\abs}[1]{\left\vert#1\right\vert}

\newcommand{\norm}[1]{\left\Vert#1\right\Vert}

\newtheorem{theorem}{Theorem}[section]
\newtheorem{lemma}[theorem]{Lemma}
\newtheorem{proposition}[theorem]{Proposition}

\newcommand \bbP{\mathbb{P}}
\newcommand \bbE{\mathbb{E}}

\def\m{\mathcal}
\def\mb{\mathbb}

\theoremstyle{plain}
\newtheorem{thm}{Theorem}[section]
\theoremstyle{plain}

\theoremstyle{remark}

\theoremstyle{plain}

\begin{document}

\title{Dirichlet-Laplace priors for optimal shrinkage}
\author{Anirban Bhattacharya, Debdeep Pati, Natesh S. Pillai, David B. Dunson }
%
\maketitle

\begin{center}
\textbf{Abstract}
\end{center}
Penalized regression methods, such as $L_1$ regularization, are routinely used in high-dimensional applications, and there is a rich literature on optimality properties under sparsity assumptions.  In the Bayesian paradigm, sparsity is routinely induced through two-component mixture priors having a probability mass at zero, but such priors encounter daunting computational problems in high dimensions.  This has motivated an amazing variety of continuous shrinkage priors, which can be expressed as global-local scale mixtures of Gaussians, facilitating computation.  In sharp contrast to the frequentist literature, little is known about the properties of such priors and the convergence and concentration of the corresponding posterior distribution. In this article, we propose a new class of  Dirichlet--Laplace (DL) priors, which possess optimal posterior concentration and lead to efficient posterior computation exploiting results from normalized random measure theory.  Finite sample performance of Dirichlet--Laplace priors relative to alternatives is assessed in simulated and real data examples.
\vspace*{.3in}


\section{Introduction}

High-dimensional data have become commonplace in broad application areas, and there is an exponentially increasing literature on statistical and computational methods for big data.  In such settings, it is well known that classical methods such as maximum likelihood estimation break down, motivating a rich variety of alternatives based on penalization and thresholding.  There is a rich theoretical literature justifying the optimality properties of such penalization approaches \cite{zhao2007model,van2008high,zhang2008sparsity,meinshausen2009lasso,raskutti2011minimax,negahban2010unified}, with fast algorithms \cite{efron2004least} and compelling applied results leading to routine use of $L_1$ regularization in particular.

The overwhelming emphasis in this literature has been on rapidly producing a point estimate with good empirical and theoretical properties.  However, in many applications, it is crucial to obtain a realistic characterization of uncertainty in the estimates of parameters and functions of the parameters, and in predictions of future outcomes.  Usual frequentist approaches to characterize uncertainty, such as constructing asymptotic confidence regions or using the bootstrap, can break down in high-dimensional settings.  For example, in regression when the number of subjects $n$ is much less than the number of predictors $p$, one cannot naively appeal to asymptotic normality and resampling from the data may not provide an adequate characterization of uncertainty.  

Most penalization approaches have a Bayesian interpretation as corresponding to the mode of a posterior distribution obtained under a shrinkage prior.  For example, the wildly popular Lasso/$L_1$ regularization approach to regression \cite{tibshirani1996regression} is equivalent to maximum {\em a posteriori} (MAP) estimation under a Gaussian linear regression model having a double exponential (Laplace) prior on the coefficients.  Given this connection, it is natural to ask whether we can use the entire posterior distribution to provide a probabilistic measure of uncertainty.  In addition to providing a characterization of uncertainty, a Bayesian perspective has distinct advantages in terms of tuning parameter choice, allowing key penalty parameters to be marginalized over the posterior distribution instead of relying on cross-validation.  In addition, by inducing penalties through shrinkage priors, important new classes of penalties can be discovered that may outperform usual $L_q$-type choices.

From a frequentist perspective, we would like to be able to choose a default shrinkage prior that leads to similar optimality properties to those shown for $L_1$ penalization and other approaches.  However, instead of showing that a particular penalty leads to a point estimator having a minimax optimal rate of convergence under sparsity assumptions, we would rather like to show that the entire posterior distribution concentrates at the optimal rate, i.e., the posterior probability assigned to a shrinking neighborhood of the true parameter value converges to one, with the neighborhood size proportional to the frequentist minimax rate.  

An amazing variety of shrinkage priors have been proposed in the Bayesian literature; however with essentially no theoretical justification for the performance of these priors in the high-dimensional settings for which they were designed.  \cite{ghosal1999asymptotic} and \cite{bontemps2011bernstein} provided conditions on the prior for asymptotic normality of linear regression coefficients allowing the number of predictors $p$ to increase with sample size $n$, with \cite{ghosal1999asymptotic} requiring a very slow rate of growth and \cite{bontemps2011bernstein} assuming $p \le n$.  These
results required the prior to be sufficiently flat in a neighborhood of the true parameter value, essentially ruling out shrinkage priors.  \cite{armagan2011generalized} considered shrinkage priors in providing simple sufficient conditions for posterior consistency in linear regression where the number of variables grows slower than the sample size, though no rate of contraction was provided. 

In studying posterior contraction in high-dimensional settings, it becomes clear that it is critical to understand several aspects of the prior distribution on the high-dimensional space, including (but not limited to) the prior concentration around sparse vectors and the implied dimensionality of the prior.  Specifically, studying the reduction in dimension induced by shrinkage priors is challenging due to the lack of exact zeros, with the prior draws being sparse in only an approximate sense.  This substantial technical hurdle has prevented any previous results (to our knowledge) on posterior concentration in high-dimensional settings for shrinkage priors.  In fact, investigating these properties is critically important not just in studying frequentist optimality properties of Bayesian procedures but for Bayesians in obtaining a better understanding of the behavior of their priors and choosing associated hyperparameters.  Without such technical handle, it becomes an art to use intuition and practical experience to indirectly induce a shrinkage prior, while focusing on Gaussian scale families for computational tractability.  Some beautiful classes of priors have been proposed by \cite{griffin2010inference,carvalho2010horseshoe,armagan2011generalized} among others, with \cite{polson2010shrink} showing that essentially all existing shrinkage priors fall within the Gaussian global-local scale mixture family.  One of our primary goals is to obtain theory that can allow evaluation of existing priors and design of novel priors, which are appealing from a Bayesian perspective in allowing incorporation of prior knowledge and from a frequentist perspective in leading to minimax optimality under weak sparsity assumptions.


\section{A new class of  shrinkage priors} \label{sec:prior_prop}

\subsection{\bf{Bayesian sparsity priors in normal means problem}}

For concreteness, we focus on the widely studied normal means problem (see, for example, \cite{donoho1992maximum,johnstone2004needles,castilloneedles} and references therein); although most of the ideas developed in this paper generalize directly to high-dimensional linear and generalized linear models. In the normal means setting, one aims to estimate a $n$-dimensional mean\footnote{Following standard practice in this literature, we use $n$ to denote the dimensionality and it should not be confused with the sample size.} based on a single observation corrupted with i.i.d. standard normal noise:
\begin{align}\label{eq:norm_means}
 y_i &= \theta_i + \epsilon_i, \quad \epsilon_i \sim \mbox{N}(0, 1), \quad 1 \leq i \leq n. 
\end{align}
Let $l_0[q; n]$ denote the subset of $\mathbb{R}^n$ given by
$$l_0[q; n] = \{ \theta \in \mathbb{R}^n ~:~ \#(1 \leq j \leq n : \theta_j \neq 0) \leq q\}.$$
For a vector $ x \in \mathbb{R}^r$, let $\norm{x}_2$ denote its Euclidean norm. If the true mean $\theta_0$ is $q_n$-sparse, i.e., $\theta_0 \in l_0[q_n; n]$, with $q_n = o(n)$, the {\em squared minimax rate} in estimating $\theta_0$ in $l_2$ norm is known \cite{donoho1992maximum} to be $2q_n \log(n/q_n)(1 + o(1))$, i.e.\footnote{Given sequences $a_n, b_n$, we denote $a_n = O(b_n)$ or $a_n \lesssim b_n$ if there exists a global constant $C$ such that $a_n \leq C b_n$ and $a_n = o(b_n)$ if $a_n/b_n \to 0$ as $n \to \infty$.}
\begin{align}\label{eq:minmax}
\inf_{\hat{\theta}} \sup_{\theta_0 \in l_0[q_n; n]} E_{\theta_0} \| \hat{\theta} - \theta_0 \|_2^2 \asymp q_n \log(n/q_n). 
\end{align}
In the above display, $E_{\theta_0}$ denotes an expectation with respect to a $\mbox{N}_n(\theta_0, \mathrm{I}_n)$ density. 
In the presence of sparsity, one thus only looses a logarithmic factor (in the ambient dimension) as a penalty for not knowing the locations of the zeroes. Moreover, \eqref{eq:minmax} implies that one only needs a number of replicates in the order of the sparsity to consistently estimate the mean. Appropriate thresholding/penalized estimators achieve the minimax rate \eqref{eq:minmax}; for a somewhat comprehensive list of {\em minimax-optimal} estimators, refer to \cite{castilloneedles}. 

For a subset $S \subset \{1, \ldots, n\}$, let $|S|$ denote the cardinality of $S$ and define $\theta_S = (\theta_j : j \in S)$ for a vector $\theta \in \mathbb{R}^n$. 
Denote $\mbox{supp}(\theta)$ to be the \emph{support} of $\theta$, the subset of $\{1, \ldots, n\}$ corresponding to the non-zero entries of $\theta$. For a high-dimensional vector $\theta \in \mathbb{R}^n$, a natural way to incorporate sparsity in a Bayesian framework is to use point mass mixture priors:
\begin{align}\label{eq:point_mass}
\theta_j \sim (1 - \pi) \delta_0 + \pi g_{\theta}, \quad j = 1, \ldots, n ,
\end{align}
where $\pi = \mbox{Pr}(\theta_j \neq 0)$, $\bbE\{ |\mbox{supp}(\theta)|  \mid \pi\} = n \pi$ is the prior guess on model size (sparsity level), and $g_{\theta}$ is an absolutely continuous density on $\mathbb{R}$. These priors are highly appealing in allowing separate control of the level of sparsity and the size of the signal coefficients.  If the sparsity parameter $\pi$ is estimated via empirical Bayes, the posterior median of $\theta$ is a minimax-optimal estimator \cite{johnstone2004needles} which can adapt to arbitrary sparsity levels as long as $q_n = o(n)$. 

In a fully Bayesian framework, it is common to place a beta prior on $\pi$, leading to a beta-Bernoulli prior on the model size, which conveys an automatic multiplicity adjustment \citep{scott2010bayes}. In a beautiful recent paper, \cite{castilloneedles} established that prior (\ref{eq:point_mass}) with an appropriate beta prior on $\pi$ and suitable tail conditions on $g_{\theta}$ leads to a minimax optimal rate of {\em posterior contraction}, i.e., the posterior concentrates most of its mass on a ball around $\theta_0$ of squared radius of the order of $q_n \log (n/q_n)$:
\begin{align}\label{eq:post_conc}
E_{\theta_0} \bbP(\norm{\theta - \theta_0}_2 < M s_n \mid y) \to 1,  \, \text{as} \, n \to \infty,
\end{align}
where $M > 0$ is a constant and $s_n^2 = q_n \log(n/q_n)$.   
\cite{naveen2013}  obtained consistency in model selection using point-mass mixture priors with appropriate data-driven hyperparameters.  

\subsection{\bf{Global-local shrinkage rules}}

Although point mass mixture priors are intuitively appealing and possess attractive theoretical properties, posterior sampling requires a stochastic search over an enormous space, leading to slow mixing and convergence \cite{polson2010shrink}. Computational issues and consideration that many of the $\theta_j$s may be small but not exactly zero has motivated a rich literature on continuous shrinkage priors; for some flavor of the vast literature refer to \citep{park2008bayesian,carvalho2010horseshoe,griffin2010inference,hans2011elastic,armagan2011generalized}.
\citep{polson2010shrink} noted that essentially all such shrinkage priors can be represented as global-local (GL) mixtures of Gaussians, 
\begin{align}\label{eq:lg}
\theta_j \sim \mbox{N}(0, \psi_j \tau), \quad \psi_j \sim f, \quad \tau \sim g,  
\end{align}
where $\tau$ controls global shrinkage towards the origin while the local scales $\{ \psi_j \}$ allow deviations in the degree of shrinkage.  If $g$ puts sufficient mass near zero and $f$ is appropriately chosen, GL priors in (\ref{eq:lg}) can intuitively approximate (\ref{eq:point_mass}) but through a continuous density concentrated near zero with heavy tails. 

GL priors potentially have substantial computational advantages over point mass priors, since the normal scale mixture representation allows for conjugate updating of $\theta$ and $\psi$ in a block. Moreover, a number of frequentist regularization procedures such as ridge, lasso, bridge and elastic net correspond to posterior modes under GL priors with appropriate choices of $f$ and $g$. For example, one obtains a double-exponential prior corresponding to the popular $L_1$ or lasso penalty if $f$ is an exponential distribution. However, unlike point mass priors (\ref{eq:point_mass}), many aspects of shrinkage priors are poorly understood, with the lack of exact zeroes compounding the difficulty in studying basic properties, such as prior expectation, tail bounds for the number of large signals, and prior concentration around sparse vectors. Hence, subjective Bayesians face difficulties in incorporating prior information regarding sparsity, and frequentists tend to be skeptical due to the lack of theoretical justification. 

This skepticism is somewhat warranted, as it is clearly the case that reasonable seeming priors can have poor performance in high-dimensional settings.  For example, choosing $\pi=1/2$ in prior (\ref{eq:point_mass}) leads to an exponentially small prior probability of $2^{-n}$ assigned to the null model, so that it becomes literally impossible to override that prior informativeness with the information in the data to pick the null model. However, with a beta prior on $\pi$, this problem can be avoided \citep{scott2010bayes}.  In the same vein, if one places i.i.d. $\mbox{N}(0, 1)$ priors on the entries of $\theta$, then the induced prior on $\norm{\theta}$ is highly concentrated around $\sqrt{n}$ leading to misleading inferences on $\theta$ almost everywhere.  Although these are simple examples, similar {\em multiplicity problems} \citep{scott2010bayes} can transpire more subtly in cases where complicated models/priors are involved and hence it is fundamentally important to understand properties of the prior and the posterior in the setting of \eqref{eq:norm_means}.

There has been a recent awareness of these issues, motivating a basic assessment of the marginal properties of shrinkage priors for a single $\theta_j$. Recent priors such as the horseshoe \cite{carvalho2010horseshoe} and generalized double Pareto \cite{armagan2011generalized} are carefully formulated to obtain marginals having a high concentration around zero with heavy tails.  This is well justified, but as we will see below, such marginal behavior alone is not sufficient; it is necessary to study the joint distribution of $\theta$ on $\mathbb{R}^n$.  With such motivation, we propose a class of Dirichlet-kernel priors in the next subsection. 

\subsection{\bf{Dirichlet-kernel priors}}

Let $\phi_0$ denote the standard normal density on $\mathbb{R}$. Also, let $\mbox{DE}(\tau)$ denote a zero mean double-exponential or Laplace distribution with density $f(y) = (2 \tau)^{-1} e^{- \abs{y}/\tau}$ for $y \in \mathbb{R}$. 

Let us revisit the global-local specification (\ref{eq:lg}). Integrating out the local scales $\psi_j$'s, \eqref{eq:lg} can be equivalently represented as a global scale mixture of a kernel $\mathcal{K}(\cdot)$,
\begin{eqnarray}\label{eq:scaledkernel}
\theta_j \stackrel{\text{i.i.d.}}{\sim} \mathcal{K}(\cdot \; , \tau), \quad \tau \sim g,
\end{eqnarray}
where $\mathcal{K}(x) = \int \psi^{-1/2} \phi_0(x/\sqrt{\psi}) g(\psi) d\psi $ is a symmetric unimodal density (or kernel) on $\mathbb{R}$ and $\mathcal{K}(x, \tau) := \tau^{-1/2} \mathcal{K}(x/\sqrt{\tau})$.
For example, $\psi_j \sim \mbox{Exp}(1/2)$ corresponds to a double-exponential kernel $\mathcal{K} \equiv \mbox{DE}(1)$, while $\psi_j \sim \mbox{IG}(1/2,1/2)$ results in a standard Cauchy kernel $\mathcal{K} \equiv \mbox{Ca}(0, 1)$. 

These traditional choices lead to a kernel which is \emph{bounded} in a neighborhood of zero. However, if one instead uses a half Cauchy prior $\psi_j^{1/2} \sim \mbox{Ca}_+(0,1)$, then the resulting horseshoe kernel \cite{carvalho2010horseshoe,carvalho2009handling} is unbounded with a singularity at zero. This phenomenon coupled with tail robustness properties leads to excellent empirical performance of the horseshoe. However, the joint distribution of $\theta$ under a horseshoe prior is understudied and further theoretical investigation is required to understand its operating characteristics. One can imagine that it concentrates more along sparse regions of the parameter space compared to common shrinkage priors since the singularity at zero potentially allows most of the entries to be concentrated around zero with the heavy tails ensuring concentration around the relatively small number of signals. 
 
 The above class of priors rely on obtaining a suitable kernel $\mathcal{K}$ through appropriate normal scale mixtures. In this article, we offer a fundamentally different class of shrinkage priors that alleviate the requirements on the kernel, while having attractive theoretical properties. In particular, our proposed class of Dirichlet-kernel (Dk) priors replaces the single global scale $\tau$ in \eqref{eq:scaledkernel} by a vector of scales $(\phi_1\tau, \ldots, \phi_n \tau)$, where $\phi = (\phi_1, \ldots, \phi_n)$ is constrained to lie in the $(n-1)$ dimensional simplex $\mathcal{S}^{n-1} = \{ x = (x_1, \ldots, x_{n})^{\T} : x_j \geq 0, \sum_{j=1}^{n} x_j = 1\}$ and is assigned a $\mbox{Dir}(a, \ldots, a)$ prior:
\begin{eqnarray}\label{eq:kd}
\theta_j \mid \phi_j, \tau \sim \mathcal{K}(\cdot \;, \phi_j \tau), \quad \phi \sim \mathrm{Dir}(a, \ldots, a).
\end{eqnarray}
In \eqref{eq:kd}, $\mathcal{K}$ is any symmetric (about zero) unimodal density with exponential or heavier tails; for computational purposes, we shall restrict attention to the class of kernels that can be represented as scale mixture of normals \cite{west1987scale}. While previous shrinkage priors in the literature obtain marginal behavior similar to the point mass mixture priors \eqref{eq:point_mass}, our construction aims at resembling the \emph{joint distribution} of $\theta$ under a two-component mixture prior. 
Constraining $\phi$ on $\mathcal{S}^{n-1}$ restrains the degrees of freedom of the $\phi_j$'s, offering better control on the number of dominant entries in $\theta$. In particular, letting $\phi \sim \mbox{Dir}(a, \ldots, a)$ for a suitably chosen $a$ allows (\ref{eq:kd}) to behave like \eqref{eq:point_mass} jointly, forcing a large subset of $(\theta_1, \ldots, \theta_n)$ to be \emph{simultaneously} close to zero with high probability.  

We focus on the Laplace kernel from now on for concreteness, noting that all the results stated below can be generalized to other choices. The corresponding hierarchical prior given $\tau$,
\begin{align}\label{eq:laplace_dir}
\theta_j \mid \phi, \tau \sim \mathrm{DE}(\phi_j \tau), \quad \phi \sim \mathrm{Dir}(a, \ldots, a),
\end{align}
is referred to as a Dirichlet--Laplace prior, denoted $\theta \mid \tau \sim \mbox{DL}_{a}(\tau)$. 

To understand the role of $\phi$, we undertake a study of the marginal properties of $\theta_j$ conditional on $\tau$, integrating out $\phi_j$. 
The results are summarized in Proposition \ref{propWG} below. 
\begin{proposition}\label{propWG}
If $\theta \mid \tau \sim \mathrm{DL}_{a}(\tau)$, then the marginal distribution of $\theta_j$ given $\tau$ is unbounded with a singularity at zero for any $a < 1$. Further, in the special case $a = 1/n$, the marginal distribution is a wrapped Gamma distribution $\mathrm{WG}(\tau^{-1},1/n)$, where $\mbox{WG}(\lambda, \alpha)$ has a density $f(x; \lambda, \alpha) \propto \abs{x}^{\alpha-1} e^{-\lambda \abs{x}}$ on $\mathbb{R}$. 
\end{proposition}
Thus, marginalizing over $\phi$, we obtain an unbounded kernel $\mathcal{K}$, so that the marginal density of $\theta_j \mid \tau $ has a singularity at 0 while retaining exponential tails. A proof of Proposition \ref{propWG} can be found in the appendix.   

The parameter $\tau$ plays a critical role in determining the tails of the marginal distribution of $\theta_j$'s. We consider a fully Bayesian framework where $\tau$ is assigned a prior $g$ on the positive real line and learnt from the data through the posterior. Specifically, we assume a $\mbox{gamma}(\lambda, 1/2)$ prior on $\tau$ with $\lambda = n a$. We continue to refer to the induced prior on $\theta$ implied by the hierarchical structure,
\begin{align}\label{eq:laplace_dir_tau}
\theta_j \mid \phi, \tau \sim \mathrm{DE}(\phi_j \tau), \quad \phi \sim \mathrm{Dir}(a, \ldots, a), \quad \tau \sim \mbox{gamma}(na, 1/2),
\end{align}
as a Dirichlet--Laplace prior, denoted $\theta \sim \mbox{DL}_a$. 

There is a recent frequentist literature on  including a local penalty specific to each coefficient. The adaptive Lasso \citep{zou2006adaptive,wang2007unified} relies on empirically estimated weights that are plugged in.  \cite{leng2010variable} instead propose to sample the penalty parameters from a posterior, with a sparse point estimate obtained for each draw.  These approaches do not produce a full posterior distribution but focus on sparse point estimates.




\subsection{\bf Posterior computation}\label{sec:postcomp}
 
The proposed class of DL priors leads to straightforward posterior computation via an efficient data augmented Gibbs sampler.  Note that the $\mbox{DL}_a$ prior \eqref{eq:laplace_dir_tau} can be equivalently represented as
\begin{align*}
\theta_j \sim \mbox{N}(0, \psi_j \phi_j^2 \tau^2), \, \psi_j \sim \mathrm{Exp}(1/2), \, \phi \sim \mbox{Dir}(a, \ldots, a), \, \tau \sim \mbox{gamma}(na, 1/2).
\end{align*}

We detail the steps in the normal means setting noting that the algorithm is trivially modified to accommodate normal linear regression, robust regression with heavy tailed residuals, probit models, logistic regression, factor models and other hierarchical Gaussian cases.  To reduce auto-correlation, we rely on marginalization and blocking as much as possible.  Our sampler cycles through (i) $\theta \mid \psi, \phi, \tau, y$, (ii) $\psi \mid \phi, \tau, \theta$, (iii) $\tau \mid \phi, \theta$ and (iv) $\phi \mid \theta$. We use the fact that the joint posterior of $(\psi, \phi, \tau)$ is conditionally independent of $y$ given $\theta$. Steps (ii) - (iv) together give us a draw from the conditional distribution of $(\psi, \phi, \tau) \mid \theta$, since 
\begin{align*}
[\psi, \phi, \tau \mid \theta] = [\psi \mid \phi, \tau, \theta] [\tau \mid \phi, \theta] [\phi \mid \theta] .
\end{align*}
 
Steps (i) -- (iii) are standard and hence not derived. Step (iv) is non-trivial and we develop an efficient sampling algorithm for jointly sampling $\phi$. Usual one at a time updates of a Dirichlet vector leads to tremendously slow mixing and convergence, and hence the joint update in Theorem \ref{thm:post_comp} is an important feature of our proposed prior; a proof can be found in the Appendix. Consider the following parametrization for the three-parameter generalized inverse Gaussian (giG) distribution: $Y \sim \mbox{giG}(\lambda, \rho, \chi)$ if $f(y) \propto y^{\lambda -1} e^{ - 0.5 (\rho y + \chi/y)}$ for $y > 0$. 
\begin{theorem}\label{thm:post_comp}
The joint posterior of $\phi \mid \theta$ has the same distribution as $(T_1/T, \ldots, T_n/T)$, where $T_j$ are independently distributed according to a $\mbox{giG}(a-1, 1, 2 | \theta_j| )$ distribution, and $T = \sum_{j=1}^n T_j$. 
\end{theorem}

The summary of each step are finally provided below.
\begin{description}
\item[\textbf{(i)}] To sample $\theta \mid \psi, \phi, \tau, y$, draw $\theta_j$ independently from a $\mbox{N}(\mu_j, \sigma_j^2)$ distribution with 
\begin{align*}
\sigma_j^2 = \{1 + 1/(\psi_j \phi_j^2 \tau^2)\}^{-1}, \quad \mu_j = \{1 + 1/(\psi_j \phi_j^2 \tau^2)\}^{-1} y.
\end{align*}

\item[\textbf{(ii)}]  The conditional posterior of $\psi \mid \phi, \tau, \theta$ can be sampled efficiently in a block by independently sampling $\psi_j \mid \phi, \theta$ from an inverse-Gaussian distribution $\mbox{iG}(\mu_j, \lambda)$ with $\mu_j = \phi_j \tau/|\theta_j|, \lambda = 1$. 


\item[\textbf{(iii)}] Sample the conditional posterior of $\tau \mid \phi, \theta$ from a $\mbox{giG}(\lambda-n, 1, 2 \sum_{j=1}^n | \theta_j |/\phi_j )$ distribution. 

\item[\textbf{(iv)}] To sample $\phi \mid \theta$, draw $T_1, \ldots, T_n$ independently with $T_j \sim \mbox{giG}(a-1, 1, 2 | \theta_j| )$ and set $\phi_j = T_j/T$ with $T = \sum_{j=1}^n T_j$. 
\end{description}

\section{Concentration properties of Dirchlet--Laplace priors}

In this section, we study a number of properties of the joint density of the Dirichlet--Laplace prior $\mbox{DL}_a$ on $\mathbb{R}^n$ and investigate the implied rate of posterior contraction \eqref{eq:post_conc} in the normal means setting \eqref{eq:norm_means}. Recall the hierarchical specification of $\mbox{DL}_a$ from \eqref{eq:laplace_dir_tau}. Letting $\psi_j = \phi_j \tau$ for $j = 1, \ldots, n$, a standard result (see, for example, Lemma IV.3  of \cite{zhou2012negative}) implies that $\psi_j \sim \mbox{gamma}(a,1/2)$ independently for $j = 1, \ldots, n$. Therefore, \eqref{eq:laplace_dir_tau} can be alternatively represented as\footnote{This formulation only holds when $\tau \sim \mbox{gamma}(na, 1/2)$ and is not true for the general $\mbox{DL}_a(\tau)$ class with $\tau \sim g$.}
\begin{eqnarray}\label{eq:mingyuan}
\theta_j \mid \psi_j \sim \mbox{DE}(\psi_j), \quad \psi_j \sim \mbox{Ga}(a, 1/2). 
\end{eqnarray}
The formulation \eqref{eq:mingyuan} is analytically convenient since the joint distribution factors as a product of marginals and the marginal density can be obtained analytically in Proposition \ref{propBes} below. The proof follows from standard properties of the modified Bessel function \cite{gradshteyn1980corrected}; a proof is sketched in the Appendix. 
\begin{proposition}\label{propBes}
The marginal density $\Pi$ of $\theta_j$ for any $1 \leq j \leq n$ is given by
\begin{align}\label{eq:dens_marg}
\Pi(\theta_j) = \frac{1}{2^{(1+a)/2}\Gamma(a)} \abs{\theta_j}^{(a-1)/2} K_{1-a} \big(\sqrt{2 \abs{\theta_j}}\big),
\end{align}
where
\begin{align*}
K_{\nu}(x) = \frac{\Gamma(\nu + 1/2)(2x)^{\nu} }{\sqrt{\pi}} \int_{0} ^{\infty} \frac{\cos t}{(t^2 + x^2)^{\nu+ 1/2}} dt
\end{align*}
is the modified Bessel function of the second kind. 
\end{proposition}
Figure \ref{fig:plot_marg} plots the marginal density \eqref{eq:dens_marg} to compare with other common shrinkage priors. 

\begin{figure}[h!]
\begin{center}
\includegraphics[width=5in]{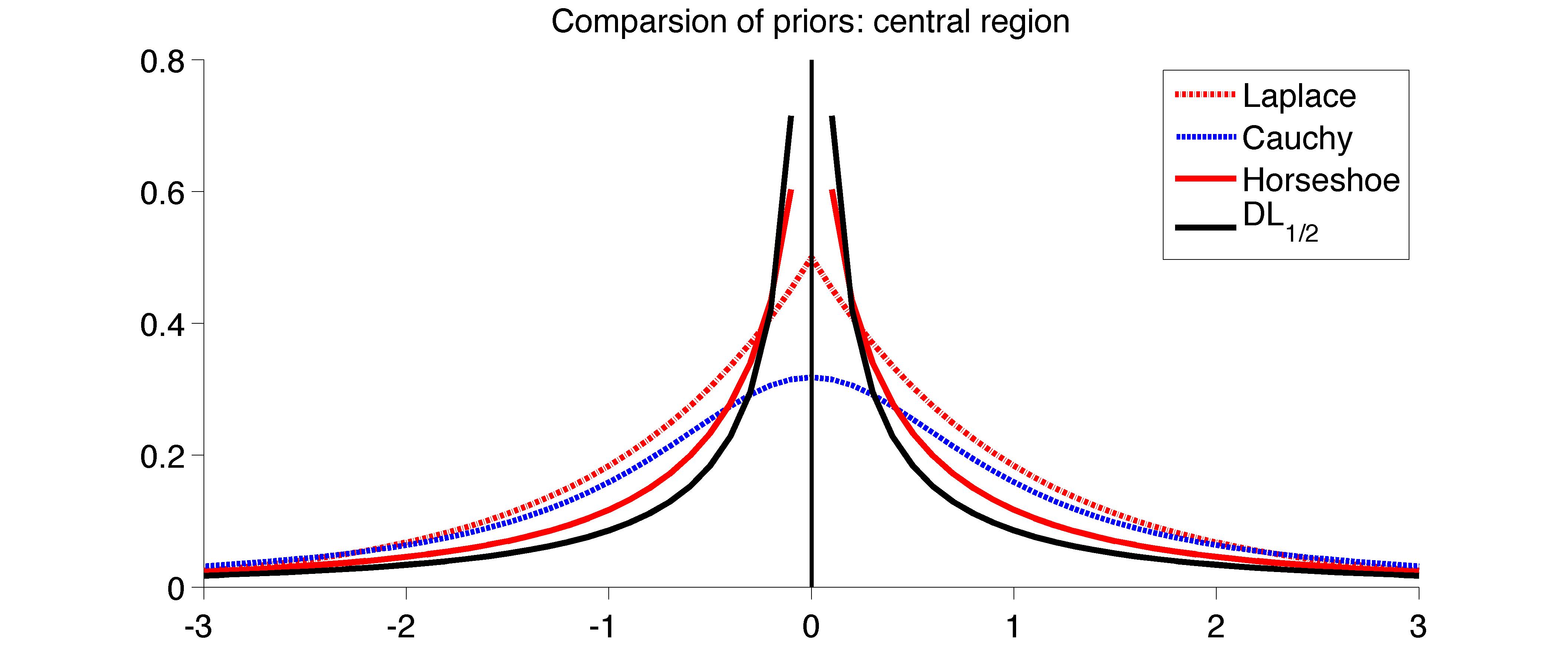} 
\includegraphics[width=5in]{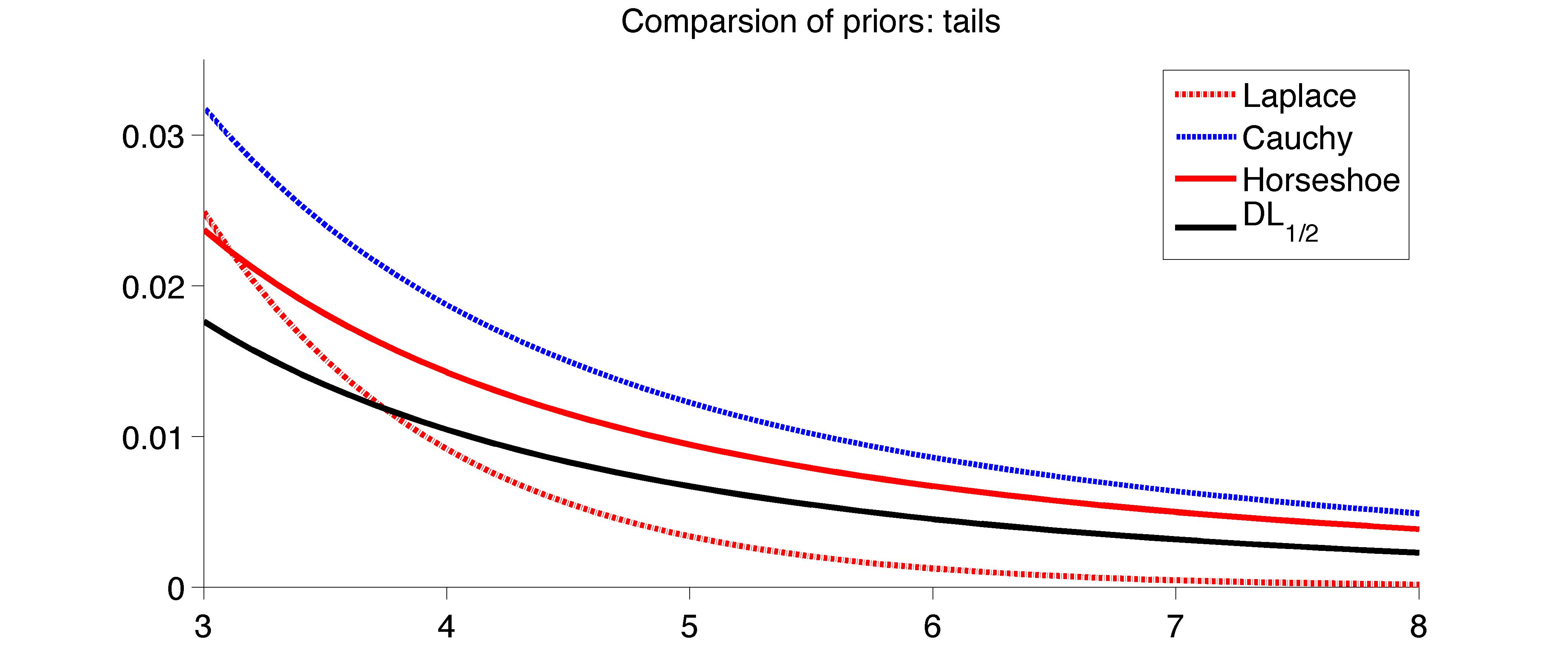}
\caption{Marginal density of the DL prior with $a = 1/2$ in comparison to other shrinkage priors. }
\end{center}
\label{fig:plot_marg}
\end{figure}

We shall continue to denote the joint density of $\theta$ on $\mb R^n$ by $\Pi$, so that $\Pi(\theta) = \prod_{j=1}^n \Pi(\theta_j)$. For a subset $S \subset \{1, \ldots, n\}$, let $\Pi_S$ denote the marginal distribution of $\theta_S = \{\theta_j : j \in S\} \in \mb R^{|S|}$. For a Borel set $A \subset \mb R^n$, let $\bbP(A) = \int_{A} \Pi(\theta) d\theta$ denote the prior probability of $A$, and $\bbP(A \mid y^{(n)})$ the posterior probability given data $y^{(n)} = (y_1, \ldots, y_n)$ and the model \eqref{eq:norm_means}. Finally, let $E_{\theta_0}/P_{\theta_0}$ respectively indicate an expectation/probability w.r.t. the $\mbox{N}_n(\theta_0, \mathrm{I}_n)$ density. We now establish that under mild restrictions on $\norm{\theta_0}$, the posterior arising from the $\mbox{DL}_a$ prior \eqref{eq:laplace_dir_tau} contracts at the minimax rate of convergence for an appropriate choice of the Dirichlet concentration parameter $a$. 
\begin{thm}\label{thm:minimaxDL}
Consider model \eqref{eq:norm_means} with $\theta \sim \mathrm{DL}_{a_n}$ as in \eqref{eq:laplace_dir_tau}, where $a_n = n^{-(1 + \beta)}$ for some $\beta > 0$ small. Assume $\theta_0 \in l_0[q_n; n]$ with $q_n = o(n)$ and $\norm{\theta_0}_2^2 \leq q_n \log^4 n$. Then, with $s_n^2 = q_n \log(n/q_n)$ and for some constant $M > 0$,
\begin{align}\label{eq:dl_cons}
\lim_{n \to \infty} E_{\theta_0} \bbP(\norm{\theta - \theta_0}_2 < M s_n \mid y) = 1. 
\end{align}
If $a_n = 1/n$ instead, then \eqref{eq:dl_cons} holds when $q_n \gtrsim \log n$. 
\end{thm}
A proof of Theorem \ref{thm:minimaxDL} can be found in Section \ref{sec:main_pf}. To best of our knowledge, Theorem \ref{thm:minimaxDL} is the first result obtaining posterior contraction rates for a continuous shrinkage prior in the normal means setting or the closely related high-dimensional regression problem. Theorem \ref{thm:minimaxDL} posits that when the parameter $a$ in the Dirichlet--Laplace prior is chosen, depending on the sample size, to be $n^{-(1+\beta)}$ for any $\beta > 0$ small, the resulting posterior contracts at the minimax rate \eqref{eq:minmax}, provided $\norm{\theta_0}_2^2 \leq q_n \log^4 n$. Using the Cauchy--Schwartz inequality, $\norm{\theta_0}_1^2 \leq q_n \norm{\theta_0}_2^2$ for $\theta_0 \in l_0[q_n; n]$ and the bound on $\norm{\theta_0}_2$ implies that $\norm{\theta_0}_1 \leq q_n (\log n)^2$. Hence, the condition in Theorem \ref{thm:minimaxDL} permits each non-zero signal to grow at a $(\log n)^2$ rate, which is a fairly mild assumption. Moreover, in a recent technical report, the authors showed that a large subclass of global-local priors \eqref{eq:lg} including the Bayesian lasso lead to a sub-optimal rate of posterior convergence; i.e., the expression in \eqref{eq:dl_cons} converges to $0$ whenever $\norm{\theta_0}_2^2/q_n \to \infty$. Therefore, Theorem \ref{thm:minimaxDL} indeed provides a substantial improvement over a large class of GL priors.

The choice $a_n = n^{-(1+\beta)}$ will be evident from the various auxiliary results in Section \ref{sec:aux}, specifically Lemma \ref{lem:prior_m} and Theorem \ref{thm:compress}.   The conclusion of Theorem  \ref{thm:minimaxDL} continues to hold when $a_n=1/n$ under an additional mild assumption on the sparsity $q_n$.  In Table \ref{tab:1} of Section \ref{sec:sims},  detailed empirical results are provided with $a_n = 1/n$ as a default choice. 

The lower bound result alluded to in the previous paragraph precludes GL priors with polynomial tails, such as the horseshoe.  We hope to address the polynomial tails case elsewhere, though based on strong empirical performance, we conjecture that the horseshoe leads to the optimal posterior contraction in a much broader domain compared to the Bayesian lasso and other common shrinkage priors. 

\subsection{\bf{Auxiliary results}}\label{sec:aux}
In this section, we state a number of properties of the DL prior which provide a better understanding of the joint prior structure and also crucially help us in proving Theorem \ref{thm:minimaxDL}.

We first provide useful bounds on the joint density of the DL prior in Lemma \ref{lem:prior_b} below; a proof can be found in the Appendix. 
\begin{lemma}\label{lem:prior_b}
Consider the $DL_{a}$ prior on $\mb R^n$ for $a$ small. Let $S \subset \{1, \ldots, n\}$ and $\eta \in \mb R^{|S|}$.

If $\min_{1\leq j \leq |S|} |\eta_j| > \delta$ for $\delta$ small, then
\begin{align}
& \log \Pi_S(\eta) \leq  C |S| \log(1/\delta), \label{eq:lip1}
\end{align}
where $C > 0$ is an absolute constant. 

If $\|\eta\|_2 \leq m$ for $m$ large, then 
\begin{align}
& -\log \Pi_S(\eta) \leq C \{ |S| \log(1/a) + |S|^{3/4} m^{1/2}\}, \label{eq:lip2}
\end{align}
where $C > 0$ is an absolute constant. 
\end{lemma}

It is evident from Figure \ref{fig:plot_marg} that the univariate marginal density $\Pi$ has an infinite spike near zero. We quantify the probability assigned to a small $\delta$-neighborhood of the origin in Lemma \ref{lem:prior_m} below. 
\begin{lemma}\label{lem:prior_m}
Assume $\theta_1 \in \mb R$ has a probability density $\Pi$ as in \eqref{eq:dens_marg}. Then, for $\delta > 0$ small,
$$
\bbP(|\theta_1| > \delta) \leq C \log(1/\delta)/\Gamma(a),
$$
where $C > 0$ is an absolute constant. 
\end{lemma}
A proof of Lemma \ref{lem:prior_m} can be found in the Appendix. 

In case of point mass mixture priors \eqref{eq:point_mass}, the induced prior on the model size $|\mbox{supp}(\theta)|$ follows a $\mbox{Binomial}(n, \pi)$ prior (given $\pi$), facilitating study of the multiplicity phenomenon \cite{scott2010bayes}.  However, $\bbP(\theta = 0) = 1$ for any continuous shrinkage prior, which compounds the difficulty in studying the degree of shrinkage for these classes of priors. Letting $\mbox{supp}_{\delta}(\theta) = \{ j : |\theta_j| > \delta\}$ to be the entries in $\theta$ larger than $\delta$ in magnitude, we propose $|\mbox{supp}_{\delta}(\theta)|$ as an approximate measure of model size for continuous shrinkage priors. We show in Theorem \ref{thm:compress} below that for an appropriate choice of $\delta$, $|\mbox{supp}_{\delta}(\theta)|$ doesn't exceed a constant multiple of the true sparsity level $q_n$ with {\em posterior probability} tending to one, a property which we refer to as {\em posterior compressibility}. 
\begin{theorem}\label{thm:compress}
Consider model \eqref{eq:norm_means} with $\theta \sim \mathrm{DL}_{a_n}$ as in \eqref{eq:laplace_dir_tau}, where $a_n = n^{-(1+\beta)}$ for some $\beta > 0$ small. Assume $\theta_0 \in l_0[q_n; n]$ with $q_n = o(n)$. Let $\delta_n = q_n / n$. Then,
\begin{align}\label{eq:suppdelta}
\lim_{n \to \infty} E_{\theta_0} \bbP(\abs{\mathrm{supp}_{\delta_n}(\theta)} > A q_n \mid y^{(n)})  = 0,  
\end{align}
for some constant $A > 0$. 
If $a_n = 1/n$ instead, then \eqref{eq:suppdelta} holds when $q_n \gtrsim \log n$. 
\end{theorem}
The choice of $\delta_n$ in Theorem \ref{thm:compress} guarantees that the entries in $\theta$ smaller than $\delta_n$ in magnitude produce a negligible contribution to $\norm{\theta}$.  Observe that the prior distribution of $|\mbox{supp}_{\delta_n}(\theta)|$ is $\mbox{Binomial}(n, \zeta_n)$, where $\zeta_n = \bbP(|\theta_1| > \delta_n)$. When $a_n = n^{-(1 + \beta)}$, $\zeta_n$ can be bounded above by $\log n/n^{1+\beta}$ in view of Lemma \ref{lem:prior_m} and the fact that $\Gamma(x) \geq 1/x$ for $x$ small. Therefore, the prior expectation $\bbE \abs{\mathrm{supp}_{\delta_n}(\theta)} \leq \log n/n^{\beta}$. This actually implies an exponential tail bound for $\bbP(\abs{\mathrm{supp}_{\delta_n}(\theta)} > A q_n)$ by Chernoff's method, which is instrumental in deriving Theorem \ref{thm:compress}. A proof of Theorem \ref{thm:compress} along these lines can be found in the Appendix. 

The posterior compressibility property in Theorem \ref{thm:compress} ensures that the dimensionality of the posterior distribution of $\theta$ (in an approximate sense) doesn't substantially overshoot the true dimensionality of $\theta_0$, which together with the bounds on the joint prior density near zero and infinity in Lemma \ref{lem:prior_b} delivers the minimax rate in Theorem \ref{thm:minimaxDL}.


\section{Simulation Study}\label{sec:sims}

To illustrate the finite-sample performance of the proposed DL prior (with $a= 1/n$), we show the results from a replicated simulation study with various dimensionality $n$ and sparsity level $q_n$.  In each setting, we have $100$ replicates of a $n$-dimensional vector $y$ sampled from a $\mbox{N}_{n}(\theta_0, \mathrm{I}_n)$ distribution with $\theta_0$ having $q_n$ non-zero entries which are all set to be a constant $A > 0$. We chose two values of $n$, namely $n = 100, 200$. For each $n$, we let $q_n  = 5, 10, 20 \%$ of $n$ and choose $A = 7, 8$. This results in $12$ simulation settings in total. The simulations were designed to mimic the setting in Section 3 where $\theta_0$ is sparse with a few moderate-sized coefficients. 

\begin{table}[htbp]
\caption{Squared error comparison  over 100 replicates. Average squared error across replicates reported for BL (Bayesian lasso), DL (Dirichlet--Laplace), LS (Lasso), EBMed (Empirical Bayes median), PM (Point mass prior) and HS (horseshoe). } \small
\begin{flushleft} 
\begin{tabular}{ccccccccccccc} \toprule
 {\bf n} & \multicolumn{6}{c}{{\bf 100}} & \multicolumn{6}{c}{{\bf 200}}  \\
 \cmidrule(lr){2-7} \cmidrule(lr){8-13} \\
 {\bf $\mathbf{\frac{q_n}{n} \%}$} & \multicolumn{2}{c}{{\bf 5}} & \multicolumn{2}{c}{{\bf 10}} & 
  \multicolumn{2}{c}{{\bf 20}} & \multicolumn{2}{c}{{\bf 5}} & \multicolumn{2}{c}{{\bf 10}} & 
  \multicolumn{2}{c}{{\bf 20}}\\
  \cmidrule(lr){2-3} \cmidrule(lr){4-5}  \cmidrule(lr){6-7} \cmidrule(lr){8-9} \cmidrule(lr){10-11} \cmidrule(lr){12-13}  \\
  {\bf A} & {\bf 7} &{\bf 8} & {\bf 7} &{\bf 8}  & {\bf 7} & {\bf 8} & {\bf 7} & {\bf 8} & {\bf 7} & {\bf 8} &{\bf 7} & {\bf 8} \\ \hline 
  BL & 33.05 & 33.63 & 49.85 & 50.04  & 68.35 & 68.54 & 64.78 & 69.34 & 99.50 & 103.15 &133.17 & 136.83 \\ 
  $DL_{1/n}$ & 8.20 & 7.19 & 17.29 & 15.35  & 32.00 & 29.40 & 16.07 & 14.28 & 33.00 & 30.80 & 65.53 & 59.61 \\  \hline
  LS & 21.25 & 19.09 & 38.68 &37.25  & 68.97 & 69.05 & 41.82 & 41.18& 75.55 & 75.12 &137.21 & 136.25 \\ 
  EBMed & 13.64 & 12.47 & 29.73 &27.96  & 60.52 & 60.22 & 26.10 & 25.52 & 57.19 & 56.05 &119.41 & 119.35 \\ 
  PM & 12.15 & 10.98 & 25.99 &24.59  & 51.36 & 50.98 &22.99 &22.26 & 49.42 &48.42 &101.54 & 101.62\\ 
  HS & 8.30 & 7.93  & 18.39 &16.27 & 37.25 & 35.18 & 15.80 & 15.09 &35.61 & 33.58 &72.15 & 70.23 \\  \hline
 
\end{tabular}
\end{flushleft}
\label{tab:1}
\end{table}

The squared error loss corresponding to the posterior median averaged across simulation replicates is provided in Table \ref{tab:1}. To offer further grounds for comparison, we have also tabulated the results for Lasso (LS), Empirical Bayes median (EBMed) as in \cite{johnstone2004needles}\footnote{The EBMed procedure was implemented using the package \cite{johnstone2005ebayesthresh}. }, posterior median with a point mass prior (PM) as in \cite{castilloneedles} and the posterior median corresponding to the horseshoe prior \cite{carvalho2010horseshoe}. For the fully Bayesian analysis using point mass mixture priors, we use a complexity prior on the subset-size, $\pi_n(s) \propto \exp \{ -\kappa s \log (2n/s) \}$ with $\kappa = 0.1$ and independent standard Laplace priors for the non-zero entries as in \cite{castilloneedles}.\footnote{Given a draw for $s$, a subset $S$ of size $s$ is drawn uniformly. Set $\theta_j = 0$ for all $j \notin S$ and draw $\theta_j, j \in S$ i.i.d. from standard Laplace. The beta-bernoulli priors in \eqref{eq:point_mass} induce a similar prior on the subset size.}

Even in this succinct summary of the results, a wide difference between the Bayesian Lasso and the proposed $\mbox{DL}_{1/n}$ is observed in Table \ref{tab:1}, vindicating our theoretical results. The horseshoe performs similarly as the $\mbox{DL}_{1/n}$. The superior performance of the $\mbox{DL}_{1/n}$ prior can be attributed to its strong concentration around the origin. However, in cases where there are several relatively small signals, the $\mbox{DL}_{1/n}$ prior can shrink all of them towards zero. In such settings, depending on the practitioner's utility function, the singularity at zero can be softened using a $\mbox{DL}_a$ prior for a larger value of $a$. 
In the next set of simulations, we report results for $a = 1/2$, whence computational gains arise as the distribution of $T_j$ in (iv) turns out to be inverse-Gaussian (iG), for which exact samplers are available.  In practice, one could as well use a discrete uniform prior on $a$ as in Section \ref{sec:real}.

\begin{table}[htbp]
\caption{Squared error comparison over 100 replicates. Average squared error for the posterior median reported for BL (Bayesian Lasso), HS (horseshoe) and DL (Dirichlet--Laplace) with $a = 1/n$ and $a = 1/2$ respectively.}
\begin{center}
\begin{tabular}{ccccccc} \toprule
 {\bf n} & \multicolumn{6}{c}{{\bf 1000}}\\
 \cmidrule(lr){1-7} \\
   {\bf A} & {\bf 2} &{\bf3} & {\bf4} &{\bf 5}  & {\bf 6} & {\bf 7}  \\ \hline 
  BL & 299.30 & 385.68 & 424.09 & 450.20  & 474.28 & 493.03\\ 
  HS &306.94 & 353.79  & 270.90 &205.43 & 182.99& 168.83  \\ 
 $ \mbox{DL}_{1/n}$ & 368.45 & 679.17 & 671.34 & 374.01  & 213.66 & 160.14  \\ 
 $ \mbox{DL}_{1/2}$ & 267.83 & 315.70 & 266.80 & 213.23  & 192.98 & 177.20 \\ \hline 
\end{tabular}
\end{center}
\label{tab:2}
\end{table}
For illustration purposes, we choose a simulation setting akin to an example in \cite{carvalho2010horseshoe}, where one has a single observation $y$ from a $n = 1000$ dimensional $\mbox{N}_{n}(\theta_0, \mathrm{I}_{n})$ distribution, with $\theta_0[1:10] = 10, \theta_0[11:100] = A$, and $\theta_0[101:1000] = 0$. We then vary $A$ from $2$ to $7$ and summarize the squared error averaged across $100$ replicates in Table \ref{tab:2}. We only compare the Bayesian shrinkage priors here; the squared error for the posterior median is tabulated. Table \ref{tab:2} clearly illustrates the need for prior elicitation in high dimensions according to the need, shrinking the noise vs. signal detection. 

For visual illustration and comparison, we finally present the results from a single replicate in the first simulation setting with $n = 200$, $q_n = 10$ and $A = 7$ in Figure \ref{fig:sima} \& \ref{fig:simb}. The blue circles indicate the entries of $y$, while the red circles correspond to the posterior median of $\theta$. The shaded region corresponds to a $95 \%$ point wise credible interval for $\theta$. 

\begin{center}
\begin{figure}[h!]
\includegraphics[width=7.5in,angle=90]{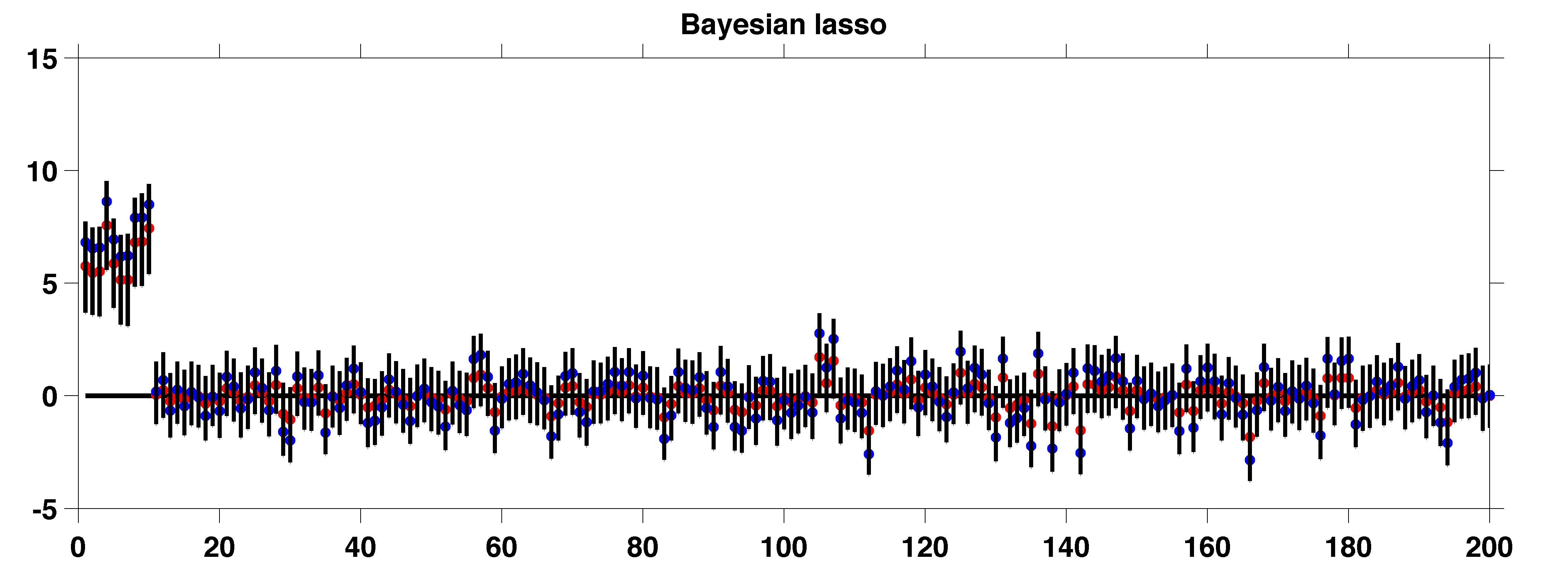} \hspace{0.6in}
\includegraphics[width=7.5in,angle=90]{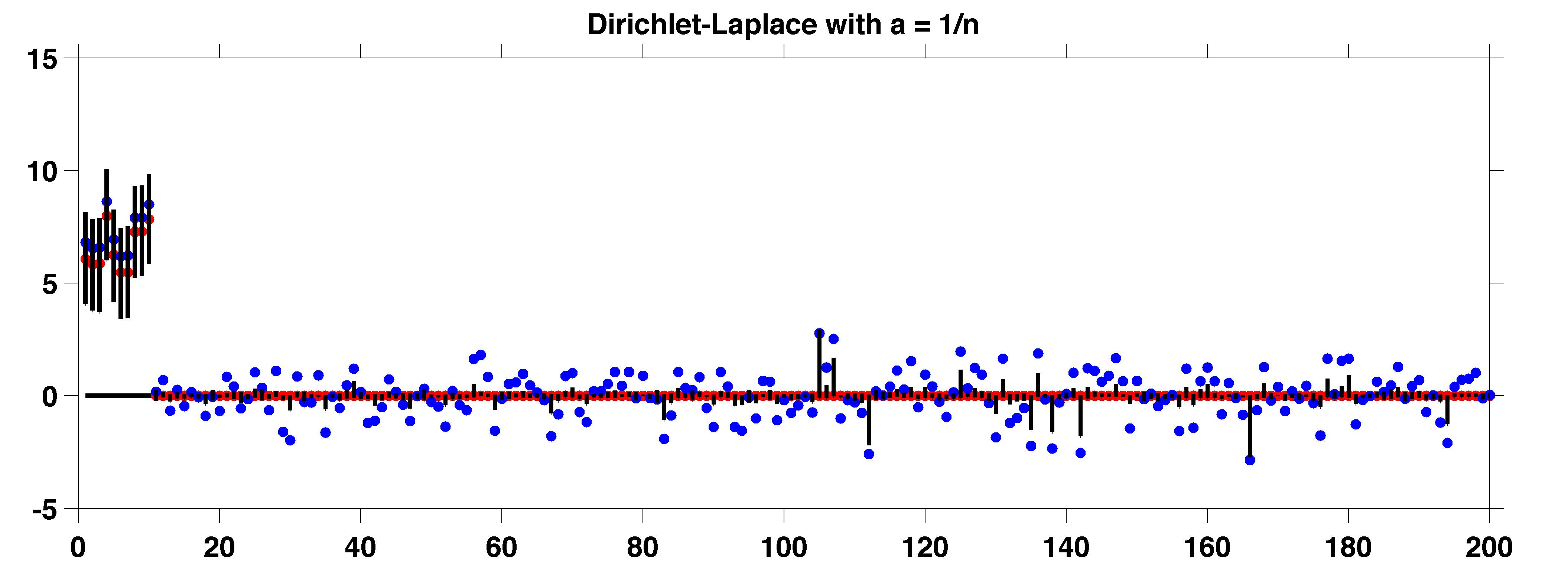}
\caption{Simulation results from a single replicate with $n=200, q_n = 10, A=7$. Blue circles = entries of $y$, red circles = posterior median of $\theta$, shaded region: $95 \%$ point wise credible interval for $\theta$. Left panel: Bayesian lasso, right panel: $\mbox{DL}_{1/n}$ prior }
\label{fig:sima}
\end{figure}
\end{center}

\begin{center}
\begin{figure}[htbp]
\includegraphics[width=7.5in,angle=90]{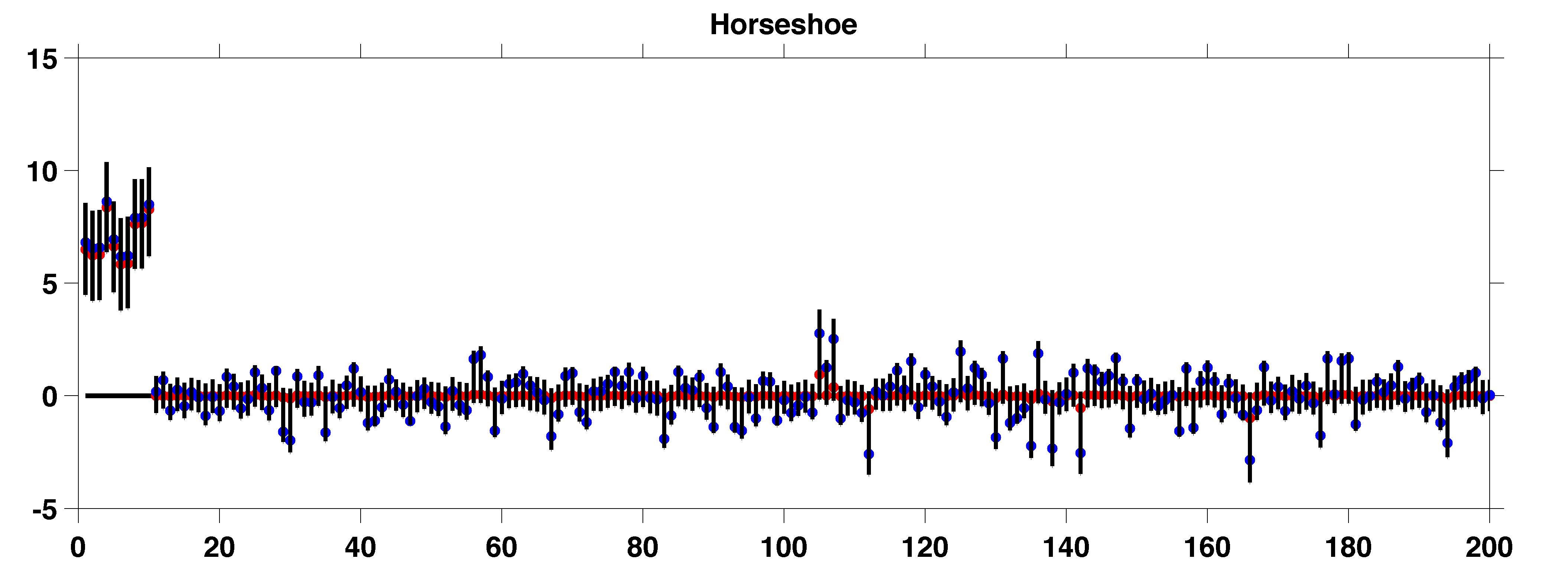}\hspace{0.6in}
\includegraphics[width=7.5in,angle=90]{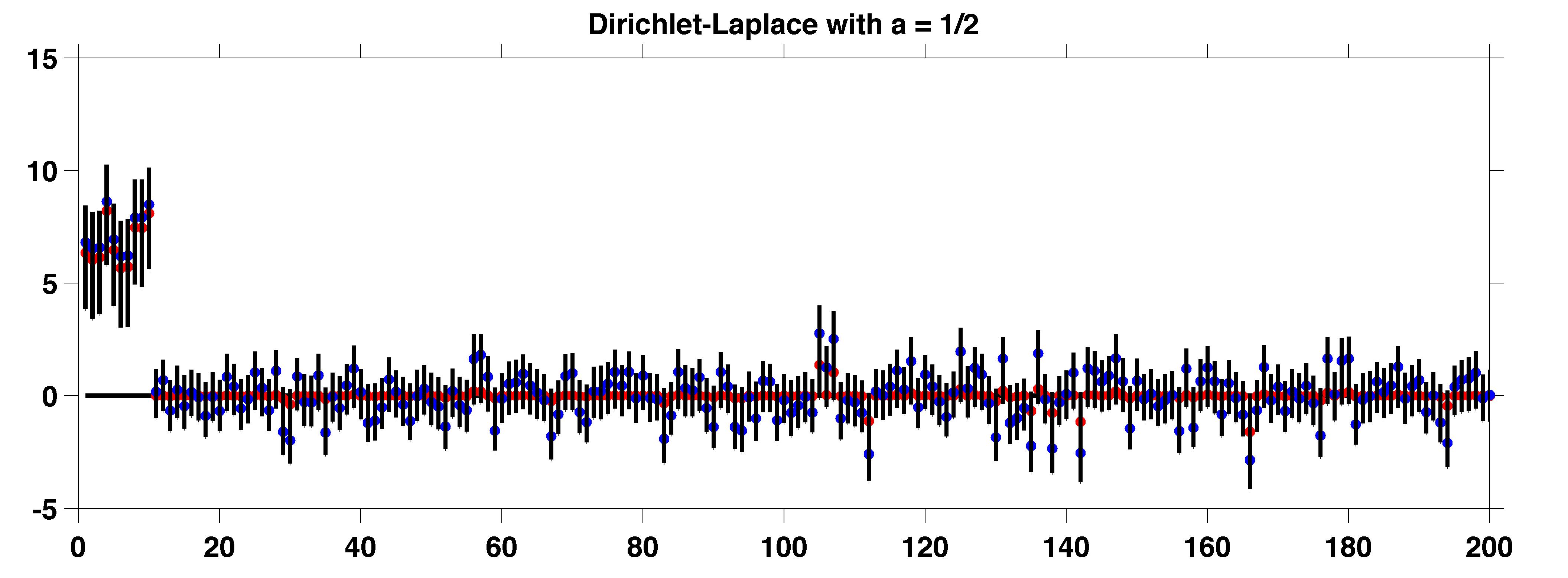} 
\caption{Simulation results from a single replicate with $n=200, q_n = 10, A=7$. Blue circles = entries of $y$, red circles = posterior median of $\theta$, shaded region: $95 \%$ point wise credible interval for $\theta$. Left panel: Horseshoe, right panel: $\mbox{DL}_{1/2}$ prior }
\label{fig:simb}
\end{figure}
\end{center}

\section{Prostate data application}\label{sec:real}

We consider a popular dataset \cite{efron2008microarrays,efron2010large} from a microarray experiment consisting of expression levels for $6033$ genes for $50$ normal control subjects and $52$ patients diagnosed with prostate cancer. The data takes the form of a $6033 \times 102$ matrix with the $(i, j)$th entry corresponding to the expression level for gene $i$ on patient $j$; the first $50$ columns correspond to the normal control subjects with the remaining $52$ for the cancer patients. 
The goal of the study is to discover genes whose expression levels {\em differ}  
between the prostate cancer patients (treatment) and normal subjects (control).  A  two sample $t$-test with $100$ degrees of freedom was implemented for each gene and the resulting t-statistic $t_i$ was converted to a $z$-statistic $z_i = \Phi^{-1} ( T_{100}(t_i))$. 
Under the null hypothesis $H_{0i}$ of no difference in expression levels between the treatment and control group for the $i$th gene, the null distribution of $z_i$ is $\mbox{N}(0, 1)$. Figure \ref{fig:hist} shows a histogram of the $z$-values, comparing it to 
a $\mbox{N}(0,1)$ density with a multiplier chosen to make the curve integrate to the same area as the histogram. The shape of the histogram suggests the presence of certain interesting genes \cite{efron2008microarrays}. 

The classical Bonferroni correction for multiple testing flags only $6$ genes as significant, while the two-group empirical Bayes method of \cite{johnstone2004needles} found $139$ significant genes, being much less conservative. The local Bayes false discovery rate (fdr) \citep{benjamini1995controlling} control method identified $54$ genes as non-null. For detailed analysis of this dataset using existing methods, refer to \cite{efron2008microarrays,efron2010large}.  

\begin{figure}[htbp]
\includegraphics[width=6.5in]{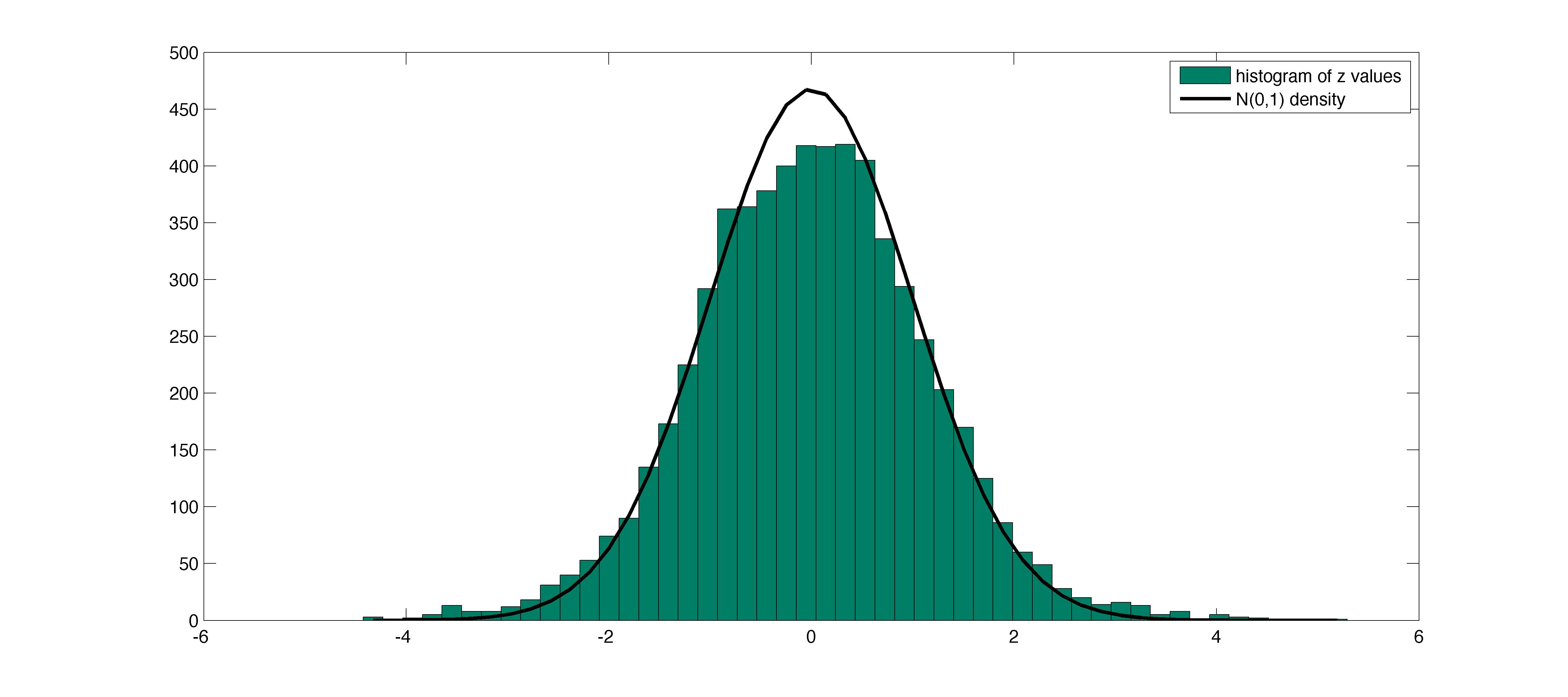} 
\caption{Histogram of z-values}
\label{fig:hist}
\end{figure}

To apply our method, we set up a normal means model $z_i = \theta_i + \epsilon_i, i=1, \ldots, 6,033$ and assign $\theta$ a $\mbox{DL}_a$ prior. Instead of fixing $a$, we use a discrete uniform prior on $a$ supported on the interval $[1/6,000, 1/2]$, with the support points of the form $10(k + 1)/6,000, k = 0, 1, \ldots, K$. Such a fully Bayesian approach allows the data to dictate the choice of the tuning parameter $a$ which is only specified up to a constant by the theory and also avoids potential numerical issues arising from fixing $a  = 1/n$ when $n$ is large. Updating $a$ is straightforward since the full conditional distribution of $a$ is again a discrete distribution on the chosen support points.  

We implemented the Gibbs sampler in  Section \ref{sec:postcomp} for 10,000 draws discarding a burn-in of 5,000.   Mixing and convergence of the Gibbs sampler was satisfactory based on examination of trace plots, with the 5,000 retained samples having an effective sample size of 2369.2 averaged across the $\theta_i$'s. The computational time per iteration scaled approximately linearly with the dimension. The posterior mode of $a$ was at $1/20$.

In this application, we expect there to be two clusters of $| \theta_i |$s, with one concentrated closely near zero corresponding to genes that are effectively not differentially expressed and another away from zero corresponding to interesting genes for further study.  As a simple automated approach, we cluster $| \theta_i |$s at each MCMC iteration using {\em kmeans} with 2 clusters.  For each iteration, the number of non-zero signals is then estimated 
by the smaller cluster size out of the two clusters. A final estimate ($M$) of the number of non-zero signals is obtained by taking the mode over all the MCMC iterations. The $M$ largest (in absolute magnitude) entries of the posterior median are identified as
 the non-zero signals.  
%


Using the above selection scheme, our method declared 128 genes as non-null. Interestingly, out of the 128 genes, 100 are common with the ones selected by EBMed. Also all the 54 genes obtained using FDR control form a subset of the selected $128$ genes. Horseshoe is overly conservative; it selected only $1$ gene (index: 610) using the same clustering procedure; the selected gene was the one with the largest {\em effect size} (refer to Table 11.2 in  \cite{efron2010large}).

\section{Proof of Theorem 3.1}\label{sec:main_pf}

We prove Theorem \ref{thm:minimaxDL} for $a_n = 1/n$ in details and note the places where the proof differs in case of $a_n = n^{-(1+\beta)}$. Recall $\theta_{S}
: = \{ \theta_j, j \in S\}$ and for $\delta \geq 0$, 
$\mathrm{supp}_{\delta}(\theta) :=  \{ j : \abs{\theta_j} > \delta \}$. Let $E_{\theta_0}/P_{\theta_0}$ respectively indicate an expectation/probability w.r.t. the 
$\mbox{N}_n(\theta_0, \mathrm{I}_n)$ density. 

For a sequence of positive real numbers $r_n$ to be chosen later, let  $\delta_n = r_n/n$.  Define $\mathcal{D}_n =  \int \prod_{i=1}^n f_{\theta_{0i}}(y_i)/f_{\theta_i}(y_i) \, d\Pi(\theta)$.
Let
\begin{eqnarray*}
\mathcal{A}_n =  \{\mathcal{D}_n \geq e^{-4r_n^2} \bbP(\norm{\theta - \theta_0}_2 \leq 2r_n) \} 
\end{eqnarray*}
be a subset of $\sigma(y^{(n)})$, the sigma-field generated by $y^{(n)}$, as in Lemma 5.2 of \cite{castilloneedles}  such that $P_{\theta_0}(\mathcal{A}_n^c) \leq e^{-r_n^2}$. 
Let $\mathcal{S}_n$ be the collection of subsets $S \subset \{1, 2, \ldots, n\}$ such that $\abs{S} \leq Aq_n$. For each such $S$ and a positive integer $j$, let  $\{\theta^{S, j, i}:  i = 1, \ldots, N_{S,j}\}$ be a $2j r_n$ net of $\Theta_{S, j, n}=\{\theta \in \mathbb{R}^n: \mathrm{supp}_{\delta_n}(\theta) = S,  2 jr_n \leq \norm{\theta - \theta_0}_2 \leq 2(j+1)r_n \}$ created as follows. Let $\{ \phi^{S, j, i} : i = 1, \ldots, N_{S,j}\}$ be a $j r_n$ net of the $|S|$-dimensional ball $\{ \| \phi - \theta_{0S} \| \leq 2(j+1) r_n\}$; we can choose this net in a way that $N_{S, j} \leq C^{|S|}$ for some constant $C$ (see, for example, Lemma 5.2 of \cite{vershynin2010introduction}). Letting $\theta^{S, j, i}_S = \phi^{S, j, i}$ and $\theta^{S, j, i}_k = 0$ for $k \in S^c$, we show this collection indeed forms a $2 j r_n$ net of $\Theta_{j,S,n}$. To that end, fix $\theta \in \Theta_{S, j, n}$. Clearly, $\| \theta_S - \theta_{0S} \| \leq 2 (j+1) r_n$. Find $1 \leq i \leq N_{S, j}$ such that $\| \theta^{S, j, i}_S - \theta_S \| \leq j r_n$. Then, 
\begin{align*}
\norm{\theta^{S, j, i} - \theta}_2^2  = \norm{\theta^{S, j, i}_S - \theta_S }_2^2  +  \norm{\theta_{S^c}}_2^2 
\leq (jr_n)^2 + (n-q_n)r_n^2/n^2 \leq 4j^2r_n^2,
\end{align*}
proving our claim. Therefore, the union of balls $B_{S, j, i}$ of radius $2 j r_n$ centered at $\theta^{S, j, i}$ for $1 \leq i \leq N_{S, j}$ cover $\Theta_{S, j, n}$. Since $E_{\theta_0} \bbP(\abs{\mathrm{supp}_{\delta_n}(\theta)} > A q_n \mid y^{(n)})  \to 0$ by Theorem \ref{thm:compress}, it is enough to work with $ E_{\theta_0} \bbP( \theta:  \norm{\theta - \theta_0}_2 > 2 M r_n, \mathrm{supp}_{\delta_n}(\theta)  \in \mathcal{S}_n  \mid y^{(n)} )$.   Using the standard testing argument for establishing posterior convergence rates (see, for example, the proof of Proposition 5.1 in \cite{castilloneedles}), we arrive at
\small
\begin{align}\label{eq:interior}
E_{\theta_0} \bbP( \theta:  \norm{\theta - \theta_0}_2 > 2 M r_n, \mathrm{supp}_{\delta_n}(\theta) 
\in \mathcal{S}_n  \mid y^{(n)} ) \leq
\sum_{S \in \mathcal{S}_1} \sum_{j \geq M} \sum_{i =1}^{N_{S,j}} 2 \sqrt{\beta_{j,S,i}} e^{- C j^2 r_n^2},
\end{align}
 \normalsize
where 
\begin{eqnarray*}
\beta_{S, j, i} = \frac{\bbP(B_{S, j, i})}{e^{-4 r_n^2} \bbP(\theta: \norm{\theta - \theta_0}_2 < 2 r_n )}.
\end{eqnarray*}
Let $r_n^2 = q_n \log n$. The proof of Theorem \ref{thm:minimaxDL} is completed by deriving an upper bound to $\beta_{S, j, i}$ in the following Lemma \ref{lem:betaSj} akin to Lemma 5.4 in \cite{castilloneedles}. 
\begin{lemma}\label{lem:betaSj}
$\log \beta_{S, j, i} \leq |S| \log(2j) + C (|S| + |S_0|) \log n + C' r_n^2$. 
\end{lemma}
\begin{proof}
\begin{align}
\beta_{S, j, i}  & \leq   \frac{\bbP\big(\theta \in \mathbb{R}^n:  \mathrm{supp}_{\delta_n}(\theta) = S, \norm{\theta_S - \tilde{\theta}_S^{S,j,i}}_2 
<  2 j r_n \big) }{e^{-4 r_n^2} \bbP(\theta \in \mathbb{R}^n: \norm{\theta - \theta_0}_2 < 2 r_n )}  \notag \\
&\leq \frac{\bbP \big(\theta \in \mathbb{R}^n: \abs{\theta}_j \leq \delta_n \, \forall \, j\in S^c, \abs{\theta_j} > \delta_n \, \forall \, j \in S, 
 \norm{\theta_S - \tilde{\theta}_S^{S,j,i}}_2 
< 2 jr_n \big) }{e^{-4 r_n^2} \bbP(\theta \in \mathbb{R}^n: \norm{\theta_{S_0} - \theta_{0S_0}}_2 < r_n, \norm{\theta_{S_0^c}} < r_n  )} \notag \\
&\leq  \frac{e^{4 r_n^2} \bbP(\abs{\theta_1} < \delta_n)^{n-|S|} \bbP \big(\abs{\theta_j} > \delta_n \, \forall \, j \in S, \norm{\theta_S - \tilde{\theta}_S^{S,j,i}}_2 < 2 jr_n \big) }{\bbP( \abs{\theta_1} 
<  \delta_n  )^{n-q_n}\bbP(\norm{\theta_{S_0} - \theta_{0S_0}}_2 < r_n) }\label{eq:betasji_b}. 
\end{align}
Next, we find an upper bound to 
\begin{eqnarray}\label{eq:rnji}
R_{S,j, i} =  \frac{\bbP\big(\abs{\theta_j} > \delta_n \, \forall \, j \in S, \norm{\theta_S - \tilde{\theta}_S^{S,j,i}}_2 
< 2 jr_n \big)}{\bbP(\norm{\theta_{S_0} - \theta_{0S_0}}_2 < r_n) }. 
\end{eqnarray}
Let $v_{q}(r)$ denote the $q$-dimensional Euclidean ball of radius $r$ centered at zero and $|v_{q}(r)|$ denote its volume. For the sake of brevity, denote $v_q = |v_q(1)|$, so that $|v_q(r)| = r^q v_q$. The numerator of \eqref{eq:rnji} can be clearly bounded above by $|v_{|S|}(2 j r_n)| \sup_{|\theta_j| > \delta_n \, \forall \, j \in S} \Pi_S(\theta_S)$. Since the set $\{\norm{\theta_{S_0} - \theta_{0S_0}}_2 < r_n\}$ is contained in the ball $v_{|S_0|}(\|\theta_{0S_0}\|_2 + r_n) = \{\| \theta_{S_0} \|_2 \leq \|\theta_{0S_0}\|_2 + r_n\}$ and $\|\theta_{0S_0}\|_2 = \| \theta_0\|_2$, the denominator of \eqref{eq:rnji} can be bounded below by 
$|v_{|S_0|}(r_n)| \inf_{v_{|S_0|}(t_n)} \Pi_{S_0}(\theta_{S_0})$, where $t_n = \| \theta_0\|_2 + r_n$. Putting together these inequalities and invoking Lemma \ref{lem:prior_b}, we have
\begin{align}
R_{S,j,i} \leq \frac{ (2 j r_n)^{|S|} \,v_{|S|} \exp\big\{C |S| \log(1/\delta_n) \big\} }{r_n^{|S_0|} \, v_{|S_0|} \exp\big\{-C (|S_0| \log n + |S_0|^{3/4} t_n^{1/2} ) \big\}  }.
\end{align}
Using $v_q \asymp (2 \pi e)^{q/2} q^{-q/2 - 1/2}$ (see Lemma 5.3 in \cite{castilloneedles}) and $r_n^2 \geq q_n = |S_0|$, we can bound $\log \{r_n^{|S|} v_{|S|}/(r_n^{|S_0|} v_{|S_0|} )\}$ from above by $C(|S| \log n + r_n^2)$. Therefore, we have
\begin{align}\label{eq:rsji_2}
\log R_{S,j,i} \leq |S| \log(2j) + C\{ |S| \log n + r_n^ 2 + |S_0| \log n + |S| \log(1/\delta_n) + |S_0|^{3/4} t_n^{1/2} \}. 
\end{align}
Now, since $\norm{\theta_0}_2^2 \leq q_n \log^{4} n$ and $r_n^2 = q_n \log n$, we have $t_n \lesssim q_n^{1/2} \log^{2} n$ and hence $|S_0|^{3/4} t_n^{1/2} \lesssim q_n \log n = r_n^2$. Substituting in \eqref{eq:rsji_2}, we have
$$
\log R_{S,j,i} \leq |S| \log(2j) + C (|S| + |S_0|) \log n + C' r_n^2. 
$$
Finally, $\bbP(|\theta_1| < \delta_n)^{n-|S|}/\bbP(|\theta_1| < \delta_n)^{n-q_n} \leq \bbP(|\theta_1| < \delta_n)^{-|S|}$. Using Lemma \ref{lem:prior_m}, $\bbP(|\theta_1| < \delta_n) \geq (1 - \log n/n)$, which implies that $ \bbP(|\theta_1| < \delta_n)^{-|S|} \leq e^{\log n}$. 
\end{proof}

Substituting the upper bound for $\beta_{S,j,i}$ obtained in Lemma \ref{lem:betaSj},  and noting that $\abs{S} \leq Aq_n$ and $\abs{N_{S,j}} \leq e^{C\abs{S}}$,  the expression in the left hand side of  \eqref{eq:interior} can be bounded above by 
\begin{align*}
& 2\sum_{S \in \mathcal{S}_1} \sum_{j \geq M} \sum_{i =1}^{N_{S,j}}  \exp\{Aq_n \log(2j)/2 + C/2 (A+1)q_n \log n + C'2 r_n^2/2\} e^{- C j^2 r_n^2} \\
& \leq 2\sum_{S \in \mathcal{S}_1} \sum_{j \geq M} \exp\{C q_n + Aq_n \log(2j)/2 + C/2 (A+1)q_n \log n + C' r_n^2/2\} e^{- C j^2 r_n^2}.  
\end{align*}
Since $\abs{\mathcal{S}_1} \leq Aq_n{n \choose Aq_n} \leq Aq_ne^{Aq_n \log (ne/Aq_n)}$, it follows that $E_{\theta_0} \bbP( \theta:  \norm{\theta - \theta_0}_2 > 2 M r_n, \mathrm{supp}_{\delta_n}(\theta) 
\in \mathcal{S}_n  \mid y^{(n)} ) \to 0$ for large $M > 0$. 

When $a_n = n^{-(1+ \beta)}$, the conclusion of Lemma \ref{lem:betaSj} remains unchanged and the proof of Theorem \ref{thm:compress} does not require $q_n \gtrsim \log n$. The rest of the proof remains exactly the same.


%

\appendix 
\section*{Appendix}
%
\subsection*{\bf Proof of Proposition \ref{propWG}}

When $a = 1/n$, $\phi_j \sim \mbox{Beta}(1/n,1-1/n)$ marginally. Hence, the marginal distribution of $\theta_j$ given $\tau$ is proportional to
\begin{align*}
\int_{\phi_j=0}^1 e^{- \abs{\theta_j}/(\phi_j \tau)} \bigg(\frac{\phi_j}{1 - \phi_j} \bigg)^{1/n} \phi_j^{-2} d \phi_j. 
\end{align*}
Substituting $z = \phi_j/(1 - \phi_j)$ so that $\phi_j = z/(1+z)$, the above integral reduces to 
\begin{align*}
e^{-\abs{\theta_j}/\tau} \int_{z=0}^{\infty} e^{- \abs{\theta_j}/(\tau z)} z^{-(2 - 1/n)} dz \propto e^{- \abs{\theta_j}/\tau} \abs{\theta_j}^{1/n-1}.
\end{align*}
In the general case, $\phi_j \sim \mbox{Beta}(a, (n-1)a)$ marginally. Substituting $z = \phi_j/(1 - \phi_j)$ as before, the marginal density of $\theta_j$ is proportional to
\begin{align*}
e^{-\abs{\theta_j}/\tau} \int_{z=0}^{\infty} e^{- \abs{\theta_j}/(\tau z)} z^{-(2 - a)}  \bigg(\frac{1}{1+z}\bigg)^{na-1} dz. 
\end{align*}
The above integral can clearly be bounded below by a constant multiple of 
\begin{align*}
e^{-\abs{\theta_j}/\tau} \int_{z=0}^{1} e^{- \abs{\theta_j}/(\tau z)} z^{-(2 - a)}  dz. 
\end{align*}
The above expression clearly diverges to infinity as $|\theta_j| \to 0$ by the monotone convergence theorem. 

%

\subsection*{{\bf Proof of Theorem \ref{thm:post_comp}}}
Integrating out $\tau$, the joint posterior of $\phi \mid \theta$ has the form
\begin{align}\label{eq:marg_post_phi}
\pi(\phi_1, \ldots, \phi_{n-1} \mid \theta) \propto \prod_{j=1}^n \bigg[ \phi_j^{a - 1} \frac{1}{\phi_j}  \bigg] 
\int_{\tau = 0}^{\infty} e^{-\tau/2} \tau^{\lambda-n-1} e^{- \sum_{j=1}^n | \theta_j | / (\phi_j \tau) } d \tau. 
\end{align}
We now state a result from the theory of normalized random measures (see, for example, (36) in \cite{kruijer2010adaptive}). Suppose $T_1, \ldots, T_n$ are independent random variables with $T_j$ having a density $f_j$ on $(0, \infty)$. Let $\phi_j = T_j/T$ with $T = \sum_{j=1}^n T_j$. Then, the joint density $f$ of $(\phi_1, \ldots, \phi_{n-1})$ supported on the simplex $\mathcal{S}^{n-1}$ has the form
\begin{align}\label{eq:normalized_rep}
f(\phi_1, \ldots, \phi_{n-1}) = \int_{t=0}^{\infty} t^{n-1} \prod_{j=1}^n f_j(\phi_j t) dt ,
\end{align}
where $\phi_n = 1 - \sum_{j=1}^{n-1} \phi_j$. Setting $f_j(x) \propto \frac{1}{x^{\delta}} e^{- | \theta_j |/ x } e^{- x/2} $ in \eqref{eq:normalized_rep}, we get
\begin{align}\label{eq:nrep1}
f(\phi_1, \ldots, \phi_{n-1})  = \bigg[ \prod_{j=1}^n \frac{1}{\phi_j^{\delta}} \bigg] \int_{t=0}^{\infty} e^{-t/2} t^{n-1 - n\delta} e^{- \sum_{j=1}^n | \theta_j | / (\phi_j t) } dt.
\end{align}
We aim to equate the expression in \eqref{eq:nrep1} with the expression in \eqref{eq:marg_post_phi}. Comparing the exponent of $\phi_j$ gives us $\delta = 2 - a$. The other requirement $n - 1 - n \delta = \lambda - n - 1$ is also satisfied, since $\lambda = na$. The proof is completed by observing that $f_j$ corresponds to a $\mbox{giG}(a-1, 1, 2 | \theta_j| )$ when $\delta = 2 - a$. 

\subsection*{{\bf Proof of Proposition \ref{propBes}}}

By \eqref{eq:mingyuan}, 
$$\Pi(\theta_j) = \frac{(1/2)^a}{2 \Gamma(a)} \int_{\psi_j = 0}^{\infty} e^{-|\theta_j|/\psi_j} \psi_j^{a-2} e^{-\psi_j/2} d\psi_j
= \frac{(1/2)^a}{2 \Gamma(a)} \int_{z=0}^{\infty} e^{-z |\theta_j|} z^{-a} e^{-2/z} dz.$$
The result follows from 8.432.7 in \cite{gradshteyn1980corrected}. 

\subsection*{{\bf Proof of Lemma \ref{lem:prior_b}}}

Letting $h(x) = \log \Pi(x)$, we have $\log \Pi_S(\eta) = \sum_{1 \leq j \leq |S|} h(\eta_j) $. 

We first prove \eqref{eq:lip1}. Since $\Pi(x)$, and hence $h(x)$, is monotonically decreasing in $|x|$, and $|\eta_j| > \delta$ for all $j$, we have $\log \Pi_S(\eta) \leq |S| h(\delta)$. Using $K_{\alpha}(z) \asymp z^{-\alpha}$ for $|z|$ small and $\Gamma(a) \asymp a^{-1}$ for $a$ small, we have from \eqref{eq:dens_marg} that $\Pi(\delta) \asymp a^{-1} |\delta|^{(a - 1)}$ and hence $h(\delta) \asymp (1 - a) \log(\delta^{-1}) - \log a^{-1} + C \leq C \log(\delta^{-1})$. 

We next prove \eqref{eq:lip2}. Noting that $K_{\alpha}(z) \gtrsim e^{-z}/z$ for $|z|$ large (section 9.7 of \cite{abramowitz1965handbook}), we have from \eqref{eq:dens_marg} that $-h(x) \leq \log a^{-1} + 3/2 \log |x| + \sqrt{2} \sqrt{|x|}$ for $|x|$ large. Using Cauchy--Schwartz inequality twice, we have $(\sum_{j=1}^{|S|} \sqrt{|\eta_j|})^4 \leq |S|^3 \|\eta\|_2^2$, which implies $\sum_{j=1}^{|S|} \sqrt{|\eta_j|} \leq |S|^{3/4} \|\eta\|_2^{1/2} \leq |S|^{3/4} m^{1/2}$. 

\subsection*{{\bf Proof of Lemma \ref{lem:prior_m}}}

Using the representation \eqref{eq:laplace_dir_tau}, we have $\bbP(|\theta_1| > \delta \mid \psi_1) = e^{-\delta/\psi_1}$, so that,
\begin{eqnarray}
\bbP( \abs{\theta_1} > \delta)  &=& \frac{(1/2)^{a}}{\Gamma(a)}\int_{0}^{\infty} e^{-\delta/x} x^{a-1} e^{-x/2} dx\nonumber \\
&=& \frac{(1/2)^{a}}{\Gamma(a)} \bigg\{ \int_{0}^{4 \delta} e^{-\delta/x} x^{a -1} e^{-x/2} dx  + 
\int_{4 \delta}^{\infty} e^{-\delta/x} x^{a -1} e^{-x/2} dx \bigg\} \nonumber\\
&\leq& \frac{(1/2)^{a}}  {\Gamma(a)} \bigg\{C  + \int_{4 \delta}^{\infty} \frac{e^{-x/2}}{x} dx  
 \bigg\} \leq \frac{(1/2)^{a}}  {\Gamma(a)} \bigg\{ C + \int_{2 \delta}^{\infty} \frac{e^{-t}}{t} dt \bigg\}, \label{eq:ig1}
\end{eqnarray}
where $C > 0$ is a constant independent of $\delta$. 
Using a bound for the incomplete gamma function from Theorem 2 of \cite{alzer1997some}, 
\begin{eqnarray} \label{eq:ig2}
\int_{2\delta}^{\infty} \frac{e^{-t}}{t} dt  \leq -\log (1 - e^{-2 \delta}) \leq  -\log(\delta),
\end{eqnarray}
for $\delta$ small. The proof is completed by noting that $(1/2)^a$ is bounded above by a constant and $C + \log(1/\delta) \leq 2 \log(1/\delta)$ for $\delta$ small enough.

\subsection*{{\bf Proof of Theorem \ref{thm:compress}}}
For $\theta \in \mathbb{R}^n$, let $f_{\theta}(\cdot)$ denote the probability density function of a $\mbox{N}_n(\theta, \mathrm{I}_n)$ distribution and $f_{\theta_i}$ denote the univariate marginal $\mbox{N}(\theta_i, 1)$ distribution. 
Let $S_0 =  \mathrm{supp}(\theta_0)$.  Since $|S_0| = q_n$,  it suffices to prove that 
\begin{eqnarray*}
\lim_{n \to \infty} E_{\theta_0}\bbP(\abs{\mathrm{supp}_{\delta_n}(\theta) \cap S_0^c } > A q_n \mid y^{(n)})   \to 0. 
\end{eqnarray*}
Let 
$\mathcal{B}_n = \{ \abs{\mathrm{supp}_{\delta_n}(\theta) \cap S_0^c } > A q_n \}$. By (\ref{eq:mingyuan}), 
$\{  \theta_i, i  \in S_0^c \}$ is independent of $\{  \theta_i, i\in S_0 \}$ conditionally on $y^{(n)}$. Hence 
\begin{eqnarray}\label{eq:postprob_bn}
\bbP(  \mathcal{B}_n  \mid y^{(n)})  =   \frac{ \int_{\mathcal{B}_n} \prod_{i \in S_0^c} \frac{f_{\theta_i}(y_i)}{f_0(y_i)} d\Pi(\theta_i)}{ \int  \prod_{i \in S_0^c} \frac{f_{\theta_i}(y_i)}{f_0(y_i)} d\Pi(\theta_i)} :=  \frac{\mathcal{N}'_n}{\mathcal{D}'_n}, 
\end{eqnarray}  
where $\m N'_n$ and $\m D'_n$ respectively denote the numerator and denominator of the expression in \eqref{eq:postprob_bn}. Observe that 
\begin{eqnarray}\label{eq:bd_postprob}
E_{\theta_0} \bbP(  \mathcal{B}_n  \mid y^{(n)})  \leq  E_{\theta_0} \bbP( \mathcal{B}_n  \mid y^{(n)}) 1_{\mathcal{A}_n'}   + P_{\theta_0}(\mathcal{A}_n'^c),
\end{eqnarray}
where $\mathcal{A}'_n$ is a subset of $\sigma(y^{(n)})$ as in Lemma 5.2 of \cite{castilloneedles} (replacing $\theta$ by  $\theta_{S_0^c}$ and $\theta_0$ by $0$) defined as
\begin{eqnarray*}
\mathcal{A}'_n =  \{\mathcal{D}'_n \geq e^{-r_n^2} \bbP(\norm{\theta_{S_0^c}}_2 \leq r_n) \},
\end{eqnarray*}
with $P_{\theta_0}(\mathcal{A}_n^c)  \leq e^{-r_n^2}$ for some sequence of positive real numbers $r_n$. We set $r_n^2 = q_n $ here. With this choice,  from \eqref{eq:bd_postprob},
\begin{eqnarray}\label{eq:bd1_postprob}
 E_{\theta_0} \bbP( \mathcal{B}_n  \mid y^{(n)})  \leq \frac{\bbP(\mathcal{B}_n)}{ e^{-r_n^2} \bbP(\norm{\theta_{S_0^c}}_2 \leq r_n)} + e^{-r_n^2}. 
\end{eqnarray}
We have $\bbP(\norm{\theta_{S_0^c}}_2 \leq r_n) \geq \bbP(|\theta_j| < r_n/\sqrt{n} \, \forall \, j \in S_0^c) = \bbP(|\theta_1| < r_n/\sqrt{n})^{n-q_n}$, with the equality following from the representation in \eqref{eq:mingyuan}. Using Lemma \ref{lem:prior_m}, $\bbP(|\theta_1| < r_n/\sqrt{n}) \geq 1 - \log n/n$, implying $\bbP(\norm{\theta_{S_0^c}}_2 \leq r_n) \geq e^{- C \log n}$. Next, clearly $\bbP(\m B_n) \leq \bbP(\abs{\mathrm{supp}_{\delta_n}(\theta)} > A q_n)$. 
As indicated in Section \ref{sec:aux}, $|\mbox{supp}_{\delta_n}(\theta)| \sim \mbox{Binomial}(n, \zeta_n)$, with $\zeta_n = \bbP(|\theta_1| > \delta_n) \leq \log n/n$ in view of Lemma \ref{lem:prior_m}.  A version of Chernoff's inequality for the binomial distribution \cite{hagerup1990guided} states that for $B \sim \mbox{Binomial}(n,\zeta)$ and $\zeta \leq a < 1$,
\begin{align}\label{eq:chernoff}
\bbP(B > an) \leq \bigg\{ \bigg(\frac{\zeta}{a}\bigg)^a e^{a-\zeta} \bigg\}^n. 
\end{align}
When $a_n = 1/n$, $q_n \geq C_0 \log n$ for some constant $C_0 > 0$.  Setting $a_n = A q_n/n$, clearly $\zeta_n < a_n$ for some $A > 1/C_0$. Substituting in \eqref{eq:chernoff}, we have $\bbP(\abs{\mathrm{supp}_{\delta_n}(\theta)} > A q_n) \leq 
e^{Aq_n\log (e \log n) - Aq_n \log A q_n }$.  Choosing $A \geq 2e/C_0$ and using the fact that $q_n \geq C_0 \log n$, we obtain $\bbP(\abs{\mathrm{supp}_{\delta_n}(\theta)} > A q_n) \leq e^{-Aq_n\log 2 }$.  Substituting the bounds for $\bbP(\m B_n)$ and $\bbP(\norm{\theta_{S_0^c}}_2 \leq r_n)$ in \eqref{eq:bd1_postprob} and choosing larger $A$ if necessary, the expression in \eqref{eq:bd_postprob} goes to zero.

If $a_n = n^{-(1+ \beta)}$, $ \zeta_n  \leq \log n/n^{1+\beta}$ in view of Lemma \ref{lem:prior_m}.  In \eqref{eq:chernoff}, set $a_n = A q_n/n$ as before. Clearly $\zeta_n < a_n$. Substituting in \eqref{eq:chernoff}, we have $\bbP(\abs{\mathrm{supp}_{\delta_n}(\theta)} > A q_n) \leq e^{-C A q_n \log n}$. Substituting the bounds for $\bbP(\m B_n)$ and $\bbP(\norm{\theta_{S_0^c}}_2 \leq r_n)$ in \eqref{eq:bd1_postprob}, the expression in \eqref{eq:bd_postprob} goes to zero. 

%
%
%

\bibliographystyle{plain}
\bibliography{cov_refs}

\end{document}